\documentclass[final]{siamart1116}

\usepackage{amsfonts, amsmath}
\usepackage{booktabs}
\usepackage{placeins}
\usepackage{xcolor}
\usepackage{breqn}
\usepackage{hyperref}

\newcommand{\TheTitle}{Efficient treatment of bilinear forms in global optimization}
\newcommand{\TheAuthors}{Marcia Fampa \and Jon Lee}

\headers{Treating bilinear forms}{M. Fampa \& J. Lee}

\title{{\TheTitle}}
\title{{Efficient treatment of bilinear forms\\ in global optimization}\thanks{Submitted to the editors \today. \funding{Supported in part by ONR grant N00014-17-1-2296.
Additionally, part of this work was done while J. Lee was visiting the
Simons Institute for the Theory of Computing. It was partially supported by the
DIMACS/Simons Collaboration on Bridging Continuous and Discrete Optimization
through NSF grant \#CCF-1740425. M. Fampa was partially supported by CNPq grant 303898/2016-0.}}}

\author{
  Marcia Fampa\thanks{Universidade Federal do Rio de Janeiro (\email{fampa@cos.ufrj.br})}
  \and
  Jon Lee\thanks{University of Michigan, Ann Arbor, MI (\email{jonxlee@umich.edu})}
  }

\ifpdf
\hypersetup{
  pdftitle={\TheTitle},
  pdfauthor={\TheAuthors}
}
\fi

\begin{document}

\maketitle

\begin{abstract}
We efficiently treat bilinear forms
%``mixed quadratic forms'' (i.e., the sum of a symmetric
%quadratic form and a bilinear form),
in the context of global optimization, by applying McCormick convexification and by extending an approach of Saxena, Bonami and Lee for symmetric
quadratic forms to bilinear forms. A key application of our work is in treating ``structural convexity'' in a symmetric quadratic form.
\end{abstract}

% REQUIRED

% REQUIRED

%\vspace{2em}
%\noindent\textit{Keywords:} global optimization, quadratic, bilinear, McCormick, convexification, disjunctive cuts
%
%\vspace{1em}
%\noindent\textit{2010 Mathematics Subject Classification:} 90C26.

\section*{Introduction}\label{sec:intro}

There is a huge literature on applications of and techniques for
handling bilinear functions; see, for example,
\cite{MR1},
\cite{MR2},
\cite{MR3},
\cite{MR4},
\cite{MR5},
\cite{MR6},
\cite{MR7},
\cite{MR8},
\cite{MR9},
\cite{MR10},
\cite{MR11},
\cite{MR12}.
We do not make any attempt to survey this vast literature.
Rather, our interest is in seeing how certain techniques
aimed at handling symmetric quadratic forms can naturally and
effectively be extended to handle bilinear forms.

We consider bilinear forms $s(x,y):=x'Ay$,
with $x\in \mathbb{R}^n$, $y\in \mathbb{R}^m$, and
$A\in \mathbb{R}^{n \times m}$.
We assume that we have or can derive reasonable box constraints on $x$ and $y$,  ${\bf a}_{x} \leq x \leq {\bf b}_{x}$ and ${\bf a}_{y} \leq y \leq {\bf b}_{y}$.

Additionally,
we will see that it is useful to also consider
symmetric quadratic forms $q(x) := x'Qx$
and $r(y):= y'Ry$, with
symmetric $Q\in \mathbb{R}^{n \times n}$ and
symmetric $R\in \mathbb{R}^{m \times m}$.
Eventually we will mostly focus
on $q(x)$, with the understanding
that whatever we say there applies also to $r(y)$.

% We can see our general quadratic form as
%\[
%f(x,y) = \left\langle Q,xx'\right\rangle + \left\langle A,xy'\right\rangle.
%\]

 We can see our bilinear form as
\[
s(x,y) = \left\langle S,xy'\right\rangle
\]
and
our symmmetric quadratic forms as
\[
q(x) = \left\langle Q,xx'\right\rangle \mbox{ and }
r(y) = \left\langle R,yy'\right\rangle~.
\]

 Now,
 we can lift to matrix space via the defining equations:
\begin{align}
\label{EW}
&W = xy'~;\\
\label{EX}
&X =   xx'~;\\
\label{EY}
&Y =   yy'~.
\end{align}
The symmetric lifting of $xx'$ and $yy'$ is completely standard,
but the non-symmetric lifting of $xy'$ is not typically carried out.

So we could model $s(x,y)+q(x)+r(y)$ as
 \[
 \langle A,W \rangle +  \langle Q,X \rangle +  \langle R,Y \rangle~,
 \]
  and focus on relaxing  \eqref{EW}, \eqref{EX} and \eqref{EY}
  in various ways.
  Similarly to \cite{FLL_CLAIO}, we could re-express $s(x,y)+q(x)+r(y)$ as
the symmetric quadratic form
\begin{align}\label{symmetrize}
		\left(x',y'\right)\left(\begin{array}{cc}
										   Q & \frac{1}{2} A \\
										   \frac{1}{2} A' & R
										   \end{array}\right)
								            \left(\begin{array}{c}
								            	   x \\
									   	   y
										   \end{array}\right).
	\end{align}
So in a formal sense, we can mathematically reduce treatment of
$s(x,y)+q(x)+r(y)$ to that of symmetric quadratic forms,
applying then the many ideas for such functions; for example,
McCormick convexification, convex relaxation via semidefinite programming (and further quadratic and linear relaxation) together with disjunctive cuts (see \cite{SBL2008}, \cite{MIQCP} and \cite{MIQCPproj}), and d.c. programming (see \cite{MR3685193}). Following this way of
thinking though, we can go the other way around. That is, given a
symmetric quadratic form $z'Tz$, with symmetric $T\in \mathbb{R}^{p \times p}$, we might seek to partition $z$ and $T$, re-writing $z'Tz$ as
\begin{align}\label{decompT}
		\left(x',y'\right)\left(\begin{array}{cc}
										   T_{11} & T_{12} \\
										   T_{12}' & T_{22}
										   \end{array}\right)
								            \left(\begin{array}{c}
								            	   x \\
									   	   y
										   \end{array}\right).
	\end{align}
Then, letting $Q:=T_{11}$~, $R:=T_{22}$~, and $A:=2T_{12}$, we get
$s(x,y)+q(x)+r(y)=z'Tz$~. Now, what could be the advantage of
viewing $z'Tz$ this way? The answer is in how we choose the partitioning
and whether we have good convexification methods related to the partitioning.
Consider a situation in which we can do this, either via a heuristic or via knowledge of the model where $T$ arose, so that $Q=T_{11}$ is large and \emph{positive semidefinite}. Further assume that our goal is to minimize a function having $z'Tz$ as a summand (or as a summand on the
left-hand side of a constraint $\cdots \leq 0$). In such a common scenario,
we need a convex under-estimator for summands, and
$q(x)$ \emph{is its own convex under-estimator}. So, following the
philosophy of \cite{MR3685193} (and other works employing d.c. programming),
we extract convexity from our functions and work to lower bound the remaining
non-convexity. In  \cite{MR3685193}, we assume that the convexity in $z'Tz$ is
\emph{completely hidden}, and we expose it algebraically (using an eigen-decomposition) by writing $T$ as the difference of
positive semi-definite matrices, $T=Q_1-Q_2$~. If $T$ has few negative
eigenvalues, then the concave quadratic $-z'Q_2z$ may be effectively treated (see \cite{MR3685193}). But it may be that there is instead ``structural convexity'' in the sense that the matrix $T$ may have a \emph{large principal sub-matrix that
is positive semi-definite}. In such a situation, it may be preferable to
decompose $T$ as \eqref{decompT}, and then view
$z'Tz$ as $s(x,y)+q(x)+r(y)$, with $Q:=T_{11}$~, $R:=T_{22}$~, and $A:=2T_{12}$~.
In this way, we have $q(x)$ as its own convex under-estimator,
and if $n$ is large, we are left with handling the bilinear $s(x,y)$,
which we will approach using methods to be presented in
\S\ref{sec:Saxena2} and \S\ref{sec:McInstead}, and the small symmetric quadratic form $r(y)$,
which we can treat with any of the available techniques for that.
In a similar vein, perhaps rather we can find a partitioning
in which $Q:=T_{11}$ and $R:=T_{22}$ both have few negative eigenvalues;
then, we might expect that we can take advantage of the ``near convexity''
of $q(x)$ and $r(y)$ (using perhaps again techniques from  \cite{MR3685193}).

For a more specific application,
we may consider the matrix equation $HAH=H$,
with data $A\in\mathbb{R}^{m\times n}$ and variable
$H\in\mathbb{R}^{n\times m}$ (see \cite{FLL_CLAIO}
for an application involving sparse pseudo-inverses).
We can see this equation as $h_{i\cdot}Ah_{\cdot j} = h_{i,j}$~,
for all $i,j$~. Or, equivalently, as the ``mixed quadratic form''
\[
\underbrace{a_{i,j} h_{j,i}^2}_{1\times 1 \mbox{ \tiny  sym. quad}}
+ \underbrace{\sum_{k:k\not=j} \sum_{\ell:\ell\not=i} a_{\ell,k} h_{k,i} h_{j,\ell}}_{ (m-1)\times (n-1) \mbox{ \tiny  bilinear}}
+  \underbrace{\sum_{k:k\not=j}a_{i,k} h_{k,i} h_{j,i}}_{ 1\times (n-1) \mbox{ \tiny bilinear}}
+  \underbrace{\sum_{\ell:\ell\not=i} a_{\ell,j} h_{j,i} h_{j,\ell}}_{ (m-1)\times 1 \mbox{ \tiny bilinear}}~,
\]
for all $i,j$~.

\section{Convexification}\label{sec:mixing}

We can relax \eqref{EW}, \eqref{EX}, and
\eqref{EY} first by standard techniques:
\begin{itemize}
\item We can employ the semi-definite inequalities
\begin{equation}\label{sd}
X\succeq xx' \mbox{ and } Y\succeq yy'~.
\end{equation}
\item We can employ McCormick inequalities derived from
the box constraints and the bilinear terms $W_{ij}=x_iy_j$~,
$X_{ij}=x_i x_j$~, and $Y_{ij}=y_i y_j$~.
\end{itemize}

Now, suppose that we have a solution $\hat{x},\hat{y},\hat{W},\hat{X},\hat{Y}$ to a relaxation. We can consider a singular-value decomposition of
$\hat{W}-\hat{x}\hat{y}'$:
\begin{equation*}
U'(\hat{W}-\hat{x}\hat{y}')V=\Sigma~.
\end{equation*}
and an eigen-value decomposition of
$\hat{X}-\hat{x}\hat{x}'$:
\begin{equation*}
Z'\left(
\hat{X}-\hat{x}\hat{x}'
\right)Z
=\Lambda~,
\end{equation*}
Also, we would employ an eigen-value decomposition of
$\hat{Y}-\hat{y}\hat{y}'$~, but at this point we
trust the reader to see that we treat $\hat{y},\hat{Y}$
in exactly the same manner as $\hat{x},\hat{X}$~.

\subsection{Treating symmetric quadratic forms via Saxena et al.}\label{sec:Saxena1}
A violation of \eqref{EX} means a non-zero eigen-value
(in $\Lambda$). We assume that our relaxation
includes \eqref{sd}, so we can assume that
we have a positive eigen-value $\lambda$, with associated eigen-vector $z$.
That is,
\begin{equation*}
z'(\hat{X}-\hat{x}\hat{x}')z=\lambda > 0,
\end{equation*}
which motivates us to look at the \emph{concave}
inequality
\begin{equation*}
\left\langle z z', X \right\rangle - (z'x)^2 \leq 0,
\end{equation*}
which is violated by
$\hat{x},\hat{X}$.
Defining $p:=z'x$
and $s:=\left\langle z z', X \right\rangle$,
we have the simple 2-dimensional concave quadratic inequality
\begin{equation*}
s - p^2 \leq 0.
\end{equation*}
Calculating bounds on $p$, we can then make a secant inequality
and apply disjunctive cuts as described in \cite{MIQCP} (also see
\cite{SBL2008,MIQCPproj}).

\subsection{Treating bilinear forms via symmetrization}\label{sec:symmetrizing}

We could re-express $s(x,y)=x'Ay$ as
the symmetric quadratic form
\begin{align}\label{ha_inprod}
		\left(x',y'\right)\left(\begin{array}{cc}
										   \textbf{0}_{n} & \frac{1}{2}A \\
										   \frac{1}{2}A' & \textbf{0}_{m}
										   \end{array}\right)
								            \left(\begin{array}{c}
								            	   x \\
									   	   y
										   \end{array}\right),
	\end{align}
but then we are dealing with an order $m+n$ matrix.
If our context really is in handling $s(x,y)+q(x)+r(y)$, and we are willing to deal with order $m+n$ symmetric matrices, we might as well then symmetrize all at once
via \eqref{symmetrize}.

\subsection{Treating bilinear forms by emulating Saxena et al.}\label{sec:Saxena2}

Next, we extend the idea of Saxena et al. to bilinear forms.

A violation of \eqref{EW} means that there is a non-zero singular-value $\sigma$ in $\Sigma$. So, for the associated columns $u$ of $U$ and $v$ of $V$, we have
\[
u'(\hat{W}-\hat{x}\hat{y}')v=\sigma \neq 0.
\]
This motivates to look at the violated valid equation
\begin{equation}
\label{svdconst}
\left\langle uv', W\right\rangle - (u'x)(v'y) = 0.
\end{equation}
We attack this by inducing separability via
\begin{equation}
\label{qjdefinition}
\begin{array}{l}
q_1:=(u'x + v'y)/2, \\
q_2:=(u'x - v'y)/2.
\end{array}
\end{equation}
So,
\[
\begin{array}{l}
u'x = q_1 + q_2, \\
v'y = q_1 - q_2,
\end{array}
\]
and then we have
\[
(u'x)(v'y) = q_1^2 - q_2^2.
\]
In this manner,
writing
\begin{equation}
\label{rdefinition}
r:=\left\langle u v', W \right\rangle,
\end{equation}
we replace \eqref{svdconst} with
\begin{eqnarray}
r - q_1^2 +q_2^2 \leq 0, \label{univ1} \\
-r + q_1^2 -q_2^2 \leq 0.  \label{univ2}
\end{eqnarray}
Then we can treat the quadratic terms of (\ref{univ1}-\ref{univ2}) via the technique of Saxena et al. \cite{MIQCP}. That is, (i) we either
leave the convex $+q_i^2$ terms as is or possibly linearize via lower-bounding tangents, and (ii)
we make secant inequalities and disjunctive cuts on the concave $-q_i^2$ terms, which requires first calculating
lower and upper bounds on the $q_i$. Note that we can either derive
bounds on the $q_i$ from the box constraints  on $x$ and $y$,
or we can get potentially better bounds by solving further (convex)
optimization problems.

Note that if we \emph{simultaneously} treat the two concave terms ($-q_i^2$) via the disjunctive technique of Saxena et al.,
we are led to a 4-way disjunction.

%So now we just have to relax the separable nonlinear expression
%\[
%z:= t_1^2-t_2^2.
%\]
%Letting
%\begin{equation}
%\label{defzi}
%z_i := t_i^2, \mbox{ for } i=1,2,
%\end{equation}
%we can
%replace $z$
%with $z_1-z_2$,
%and then we can treat \eqref{defzi} via the technique of Saxena et al. \cite{MIQCP}. That is, we enforce the convex $t_i^2\leq z_i$ directly (or via a linearization), and then make secant inequalities and disjunctions on the concave $t_i^2\geq z_i$ (which requires calculating
%lower and upper bounds on the $t_i$).

\subsection{Treating bilinear forms with McCormick}\label{sec:McInstead}

Another possible way of relaxing \eqref{svdconst} is to apply McCormick convexification. Let
\begin{equation}
\label{defining}
\left\{
\begin{array}{l}
s:=\left \langle uv', W \right \rangle ~,\\
p_1:=u'x~,\\
p_2:=v'y~.\\
\end{array}
\right.
\end{equation}
We first calculate bounds $[a_i,b_i]$, for $p_i$ ($i=1,2$). Then we carry out the associated McCormick relaxation of $s=p_1p_2$:
\begin{align}
		& s\leq b_2 p_1 + a_1 p_2 - a_1 b_2 \label{mc1} \tag{I.1}\\
		& s\leq a_2 p_1 + b_1 p_2 - a_2 b_1 \label{mc2} \tag{I.2}\\
		& s\geq a_2 p_1 + a_1 p_2 - a_1 a_2  \label{mc3} \tag{I.3}\\
		& s\geq b_2 p_1 + b_1 p_2 - b_1 b_2  \label{mc4} \tag{I.4}
\end{align}
Substituting back in \eqref{defining},
we obtain
\begin{align}
		& \left \langle W,uv' \right \rangle\leq b_2 u'x + a_1 v'y - a_1 b_2 \label{mc1p} \tag{$\mbox{I.1}'$}\\
		& \left \langle W,uv' \right \rangle\leq a_2 u'x + b_1 v'y - a_2 b_1 \label{mc2p} \tag{$\mbox{I.2}'$}\\
		& \left \langle W,uv' \right \rangle\geq a_2 u'x + a_1 v'y - a_1 a_2  \label{mc3p} \tag{$\mbox{I.3}'$}\\
		& \left \langle W,uv' \right \rangle\geq b_2 u'x + b_1 v'y - b_1 b_2  \label{mc4p} \tag{$\mbox{I.4}'$},
\end{align}
and
we can hope that these inequalities are violated by $\hat{W}$, $\hat{x}$, $\hat{y}$.

Backing up a bit to compare with the Saxena et al. treatment of symmetric quadratic forms, here we are relaxing $s=p_1p_2$. If $p_1=p_2=:p$ (the Saxena et al. case), then we have $s=p^2$, whereupon we can distinguish the two ``sides'':
\begin{align}
		& s\geq p^2  \label{conv1} \tag{convex}\\
		& s\leq p^2 \label{conc1} \tag{concave}
\end{align}
Then Saxena et al. use (i) the convex side directly (or a linearization of it), and (ii) disjunctive programming on the concave side.

The question now begs, can we take (\ref{mc1}--\ref{mc4})  and do disjunctive programming in some nice way? It is convenient to work with box domains, so we could pick $\eta_i$ in $[a_i,b_i]$, for $i=1,2$. Then we get four boxes, by pairing one of
\[
[a_1,\eta_1], \; [\eta_1,b_1],
\]
and one of
\[
[a_2,\eta_2], \; [\eta_2,b_2].
\]
For each box, we get a new McCormick convexification (in the spirit of \ref{mc1}--\ref{mc4}). And so, as in the technique of \S\ref{sec:Saxena2}, we have a 4-way disjunction to base a disjunctive cut upon.

\section{Experiments}\label{sec:exper}

To test and compare the different approaches discussed in this paper, we consider the global optimization problem
\begin{align*}\tag{I}
& \min\ s(x,y)+q(x)+r(y)~,  \\
&{\bf a}_{x} \leq x \leq {\bf b}_{x}~,\\
&{\bf a}_{y} \leq y \leq {\bf b}_{y}~,\\
&x\in \mathbb{R}^n \ y\in \mathbb{R}^m~,
\end{align*}
where $s(x,y):=x'Ay$, $q(x):=x'Qx$, and $r(y):=y'Ry$,
 with $A\in \mathbb{R}^{n \times m}$, and with
  $Q\in \mathbb{R}^{n \times n}$ and
 $R\in \mathbb{R}^{m \times m}$, both symmetric and positive semidefinite.

The reason for considering $Q$ and $R$ to be positive semidefinite in our numerical  experiments is our main goal of investigating how the different ways of relaxing the bilinear term, $s(x,y)$, compare to each other, while taking advantage of the remaining convexity in the objective function. In case the quadratic functions are not convex, however, we can also apply the standard symmetric lifting (\ref{EX}-\ref{EY}) to relax them, as mentioned above.

We used  a set of 64 randomly generated instances for our experiments. To investigate how, or if, our results are affected by different parameters, we have generated the instances in the following way: for each pair $(m,n)$ considered, 8 instances were generated. The matrix $A$ has density of 50\% on half of them and of 100\% on the other half. In each group of 4 instances where $A$ has the same density, we have the ranks of $Q$ and $R$ set to, respectively, $25,50,75$, and $100$ percent of $n$ and $m$.

To evaluate the quality of the bounds, we solved the instances with CPLEX. We used default settings in CPLEX
v12.7.1.0 with a time limit of 1800 seconds. Denoting the best solution value obtained by CPLEX by $\bar{z}$, we compute relative duality gaps for each of our methods by
$$\mbox{Relative Duality Gap}:=\frac{\bar{z}-lb}{|\bar{z}|}\times 100\%,$$ where $lb$ is the lower bound given by each of the four approaches described in the following, namely,   \textbf{S.Mc}, \textbf{B.Mc}, \textbf{B.Mc.Disj}, \textbf{B.Mc.ExtDisj}.

The numerical experiments were conducted on a 3.20 GHz Intel Xeon ES-2667 CPU, with 128 GB RAM,
running under Windows 2012R2. We coded our algorithms in Matlab R2017b, and used MOSEK 8.0 to solve all the quadratic and linear relaxations, and also the linear programs solved to generate cuts.

Initially, we  consider the two basic ideas discussed on how to handle the bilinear term in the objective function, which lead to the relaxations (\textbf{S.Mc}) and (\textbf{B.Mc}) introduced next.

\begin{itemize}
\item \textbf{S.Mc} (Symmetric McCormick):
For constructing the first relaxation of (I), we consider
the symmetric quadratic form \eqref{symmetrize}, and linearize the objective function of the problem.

Defining
\begin{align}\label{matsym}
		\Gamma:=\left(\begin{array}{cc}
										   Q & \frac{1}{2} A \\
										   \frac{1}{2} A' & R
										   \end{array}\right), \,\, a_h:=\left(\begin{array}{c}
								         a_x \\
									   	   a_y
										   \end{array}\right), \,\, b_h:=\left(\begin{array}{c}
								         b_x \\
									   	   b_y
										   \end{array}\right),
\end{align}								       														
using the symmetric lifting
 \begin{align}\label{symlift}
H= hh', \, \mbox{where} \, h:=\left(\begin{array}{c}
								         x \\
									   	   y
										   \end{array}\right),
\end{align}
and including the McCormick inequalities to relax the matrix equation $H=hh'$, we obtain the following standard linear programming relaxation of (I):
\begin{align*}\tag{\textbf{S.Mc}}
& \min\  \langle \Gamma,H \rangle~,  \\
& H_{i,j}-b_{h_j} h_i-a_{h_i} h_j+a_{h_i} b_{h_j} \leq 0~; & \forall i,j\in \{1,\ldots,n+m\}~, \\
& H_{i,j}-a_{h_j} h_i-b_{h_i} h_j+b_{h_i} a_{h_j} \leq  0~; &\forall i,j\in \{1,\ldots,n+m\}~, \\
& H_{i,j}-a_{h_j} h_i-a_{h_i} h_j+a_{h_i} a_{h_j} \geq  0~; &\forall i,j\in \{1,\ldots,n+m\}~, \\
& H_{i,j}-b_{h_j} h_i-b_{h_i} h_j+b_{h_i} b_{h_j} \geq  0~; &\forall i,j\in \{1,\ldots,n+m\}~, \\
&{\bf a}_{h} \leq h \leq {\bf b}_{h}~,\\
&h\in \mathbb{R}^{n+m}~, H\in \mathbb{R}^{(n+m)\times (n+m)}~.
\end{align*}

%\item \textbf{Sym-MC-Disj}
%Finally, we obtain the last relaxation by iteratively adding the disjunctive cuts proposed in \cite{MIQCP} to ($I_{R2}$).
%As done for relaxation BIL-MC-Disj, we execute $10$ iterations adding disjunctive cuts. In each iteration, we add up to 4 cuts, corresponding to the most negative eigenvalues of $\hat{H}-\hat{h}\hat{h}'$, where $(\hat{h}, \hat{H})$ is the solution of the last relaxation solved.
\item  \textbf{B.Mc} (Bilinear McCormick):
Our second relaxation is a  convex quadratic program, obtained by keeping the quadratic terms of the objective function of (I), and linearizing the bilinear term using the non-symmetric lifting \eqref{EW}. To relax the matrix equation $W=xy'$, we also use McCormick inequalities derived from the box constraints.
\begin{align*}\tag{\textbf{B.Mc}}
& \min\  \langle A,W \rangle+q(x)+r(y)~,  \\
& W_{i,j}-b_{y_j} x_i-a_{x_i} y_j+a_{x_i} b_{y_j} \leq 0~; & \forall i\in N~,~j\in M~, \\
& W_{i,j}-a_{y_j} x_i-b_{x_i} y_j+b_{x_i} a_{y_j} \leq  0~; &\forall i\in N~,~j\in M~, \\
& W_{i,j}-a_{y_j} x_i-a_{x_i} y_j+a_{x_i} a_{y_j} \geq  0~; &\forall i\in N~,~j\in M~, \\
& W_{i,j}-b_{y_j} x_i-b_{x_i} y_j+b_{x_i} b_{y_j} \geq  0~; &\forall i\in N~,~j\in M~, \\
&{\bf a}_{x} \leq x \leq {\bf b}_{x}~,\\
&{\bf a}_{y} \leq y \leq {\bf b}_{y}~,\\
&x\in \mathbb{R}^n~, \ y\in \mathbb{R}^m~, W\in \mathbb{R}^{n\times m}~.
\end{align*}

\end{itemize}

Table \ref{tab:tab1} presents the average relative duality gaps obtained when solving  (\textbf{S.Mc}) and (\textbf{B.Mc}), for each group of instances of the same size, and also the average computational time to solve them. For our instances, we see a clear advantage in using the convex quadratic relaxation (\textbf{B.Mc}) over the linear relaxation (\textbf{S.Mc}). The gaps obtained by the quadratic relaxation were considerably smaller for all test instances, and the average solution time was also smaller for 6 out of 8 groups of instances. The results show that keeping the convex quadratic terms in the objective function of the relaxation is very effective for strengthening it. Also, linearizing the complete objective function does not compensate in terms of the computational effort to solve the relaxation, because of the increase in its dimension due to the symmetric lifting.

\begin{table}[ht]
\centering
\begin{tabular}{cc|cc|cc}
\hline
&&\multicolumn{2}{c|}{Relative Duality}&\multicolumn{2}{c}{Time}\\
&&\multicolumn{2}{c|}{Gap (\%)}&\multicolumn{2}{c}{(sec)}\\
$n$&$m$& B.Mc & S.Mc& B.Mc & S.Mc\\ \hline
20&4&36.36&159.58&0.49&0.60\\
20&8&62.89&165.20&0.42&0.55\\
20&16&83.24&153.04&0.51&0.97\\
20&20&106.37&184.69&0.60&0.91\\
100&4&86.79&1938.62&4.38&1.63\\
100&20&206.75&904.29&1.01&1.59\\
100&40&212.30&596.03&1.15&1.78\\
100&80&211.51&464.98&3.41&2.61\\
\hline
\end{tabular}
\caption{Comparing relaxations for box-constrained nonconvex quadratic problems}\label{tab:tab1}
\end{table}

The next goal of our experiments was to compare the two types of disjunctive cuts proposed to strengthen the relaxation of the non-symmetric lifting $W=xy'$. The disjunctive cuts are generated by solving linear programs in cutting-plane algorithms. For details on the formulation of the Cut Generating Linear Program (CGLP) in the context of disjunctive programming,  the reader is referred to \cite{MIQCP}.

We have implemented two cutting-plane algorithms, both start by solving (\textbf{B.Mc}) and obtaining its solution $(\hat{x},\hat{y}, \hat{W})$. Let  $u^i,v^i$ be the  left- and right-singular vectors
of the matrix $\hat{M}:=\hat{W}-\hat{x}\hat{y}'$, corresponding to its $i^{th}$ largest singular-value, for all $i=1,\ldots,n.cuts$. These singular vectors are used to generate $n.cuts$ disjunctive cuts to be included in the relaxation. The process iterates until the maximum number of cuts, $max.n.cuts$,  is added to the relaxation or until no cut for the solution of the current relaxation is found. In our experiments, we set   $max.n.cuts=40$ and $n.cuts=\min\{4,\sigma_+\}$, where $\sigma_+$ is the number of non-zero singular-values of $\hat{M}$.
%We also limit the CPU time of the algorithms in 3600 seconds.

In the following we distinguish our two cutting-plane algorithms.

\begin{itemize}

\item \textbf{B.Mc.Disj} (Bilinear McCormick with Saxena et al.'s Disjunctive cuts):
Here we apply what is discussed in \S\ref{sec:Saxena2}, extending the idea of the disjunctive cuts proposed by Saxena et al. in \cite{MIQCP} to the bilinear case. We consider the following convex quadratic secant inequalities for (\ref{univ1}-\ref{univ2}):
\begin{eqnarray}
 u^{i'}Wv^i + ((u^{i'}x - v^{i'}y)/2)^2 \leq (l_{q^i_1}+u_{q^i_1})((u^{i'}x + v^{i'}y)/2) - l_{q^i_1} u_{q^i_1}~, \label{sec1} \\
-u^{i'}Wv^i + ((u^{i'}x + v^{i'}y)/2)^2 \leq (l_{q^i_2}+u_{q^i_2})((u^{i'}x - v^{i'}y)/2) - l_{q^i_2} u_{q^i_2}~, \label{sec2}
\end{eqnarray}
for $i=1,\ldots,n.cuts$, where  $l_{q^i_j},u_{q^i_j}$ are bounds on $q^i_j$ ($j=1,2$),  defined as in \eqref{qjdefinition}, but replacing $u,v$ with $u^i,v^i$. The bounds are derived from the box constraints on $x$ and $y$. By splitting the ranges $[l_{q^i_j},u_{q^i_j}]$ ($j=1,2$), into two intervals each, a 4-way disjunction is derived, and used to generate disjunctive cuts.

For the formulation of the CGLP, we should consider the linear constraints of (\textbf{B.Mc}), a linear relaxation of (\ref{sec1}-\ref{sec2}), and possibly more valid linear inequalities to strengthen the relaxation of $W=xy'$. Nevertheless, there is a natural trade-off between the number of constraints used to strengthen this relaxation  (and consequently, the effectiveness of the cut generated), and the computational effort to solve the CGLP. To formulate the CGLP in our algorithm, we have considered the constraints of (\textbf{B.Mc}) and the linearization of the constraints (\ref{sec1}-\ref{sec2}), obtained by replacing the convex quadratics in their left-hand side with their tangents, computed only at the solution of the last relaxation solved.
We also have considered the linearization of valid inequalities,  similar to (\ref{sec1}-\ref{sec2}), but where the singular vectors $u_i$ and $v_i$ are replaced by the unit vectors $e_i$ and $-e_i$ in corresponding dimensions, for each $i$. As we have two inequalities for each pair $(u^i,v^i)$ in (\ref{sec1}-\ref{sec2}), the use of all the unit vectors leads to the linearization of $8nm$ inequalities, which are also considered in the formulation of the CGLP.

%computed at $n.points$ equally distributed points in the interval defined by the box constraints on $x$ and $y$; and also with the tangent computed on the solution of the last relaxation solved. We had $n.points=0$ in the tests, i.e., each quadratic is replaced with only one tangent, computed  at the solution of the last relaxation solved.

\item \textbf{B.Mc.ExtDisj}  (Bilinear McCormick with our Extended McCormick Disjunctive cuts):
Here we apply what is discussed in \S\ref{sec:McInstead}, extending the idea under Saxena et al.'s disjunctive cuts to generate disjunctive cuts from the  inequalities (\ref{mc1p}--\ref{mc4p}).
The formulation of the CGLP takes into account the constraints of (\textbf{B.Mc}), and the linear inequalities
\begin{align}
		& \left \langle W,u^{i}v^{i'} \right \rangle\leq b^{i}_2 u^{i'}x + a^{i}_1 v^{i'}y - a^{i}_1 b^{i}_2 \label{mc1pa} \tag{$\mbox{I.1a}'$}\\
		& \left \langle W,u^{i}v^{i'} \right \rangle\leq a^{i}_2 u^{i'}x + b^{i}_1 v^{i'}y - a^{i}_2 b^{i}_1 \label{mc2pa} \tag{$\mbox{I.2a}'$}\\
		& \left \langle W,u^{i}v^{i'} \right \rangle\geq a^{i}_2 u^{i'}x + a^{i}_1 v^{i'}y - a^{i}_1 a^{i}_2  \label{mc3pa} \tag{$\mbox{I.3a}'$}\\
		& \left \langle W,u^{i}v^{i'} \right \rangle\geq b^{i}_2 u^{i'}x + b^{i}_1 v^{i'}y - b^{i}_1 b^{i}_2  \label{mc4pa} \tag{$\mbox{I.4a}'$},
\end{align}
for $i=1,\ldots,n.cuts$, where  $a_j{^i},b{_j^i}$ are bounds on $p^i_j$ ($j=1,2$),  defined as in \eqref{defining}, but replacing $u,v$ with $u^i,v^i$. The bounds are derived from the box constraints on $x$ and $y$. By splitting now, the ranges $[a_j{^i},b{_j^i}]$ ($j=1,2$), into two intervals each, another 4-way disjunction is derived, and is used to generate alternative disjunctive cuts.
\end{itemize}

Table \ref{tab:tab2} presents the average relative duality gaps for each group of instances of the same size, obtained by the cutting-plane algorithms \textbf{B.Mc.Disj} and \textbf{B.Mc.ExtDisj}, and their average computational time.

For the instances used, we see a slight superiority of the bounds obtained by the cutting-plane algorithm  \textbf{B.Mc.ExtDisj} over \textbf{B.Mc.Disj}, for 7 out of 8 groups of instances. Furthermore, the computational time of \textbf{B.Mc.ExtDisj} becomes much smaller as the size of the instances increases. The reason for the faster increase in the times for \textbf{B.Mc.Disj} is the use of the $8nm$ inequalities derived from the unit vectors, as described above. However, if those inequalities are not used, the bounds obtained by   \textbf{B.Mc.Disj} become worse than the ones obtained by \textbf{B.Mc.ExtDisj}.

\begin{table}[ht]
\centering
\begin{tabular}{cc|cc|cc}
\hline
&&\multicolumn{2}{c|}{Relative Duality}&\multicolumn{2}{c}{Time}\\
&&\multicolumn{2}{c|}{Gap (\%)}&\multicolumn{2}{c}{(sec)}\\
$n$&$m$& B.Mc.Disj & B.Mc.ExtDisj& B.Mc.Disj & B.Mc.ExtDisj\\ \hline
20&4&23.23	&19.87&	49.25&	51.55\\
20&8&46.28	&44.43&	56.74&	61.85\\
20&16&65.46	&65.18&	75.97&	73.86\\
20&20&88.18	&87.63&	107.71&	95.50\\
100&4&77.53	&70.10&	157.91&	128.16\\
100&20&199.43&	193.44&	527.68&	319.15\\
100&40&196.91&	200.60&	1676.57&	678.32\\
100&80&206.64&	205.30&	2348.45&	1587.26\\
\hline
\end{tabular}
\caption{Comparing disjunctive cuts for box-constrained nonconvex quadratic problems}\label{tab:tab2}
\end{table}

Figures \ref{plots20} and \ref{plots100} also depict comparisons of the bounds obtained with our two basic relaxations  (\textbf{S.Mc}), (\textbf{B.Mc}), and with both cutting-plane algorithms, \textbf{B.Mc.Disj}, \textbf{B.Mc.ExtDisj}. But now, we have separate results for different instance configurations. In Figure \ref{plots20}, we show results for the instances with $n=20$ and $m=4,8,16,20$. For each value of $m$, the plots on the left compare the relative duality gaps for the different values of  (rank($Q$),rank($R$)), and the plots on the right, compare them for each density of $A$. In Figure \ref{plots100} we have the same comparisons for the instances with $n=100$ and $m=4,20,40,80$.

\begin{figure}
\centering
 \hspace*{-1.5cm} 
 \includegraphics[height=21.5cm, width=16.5cm]{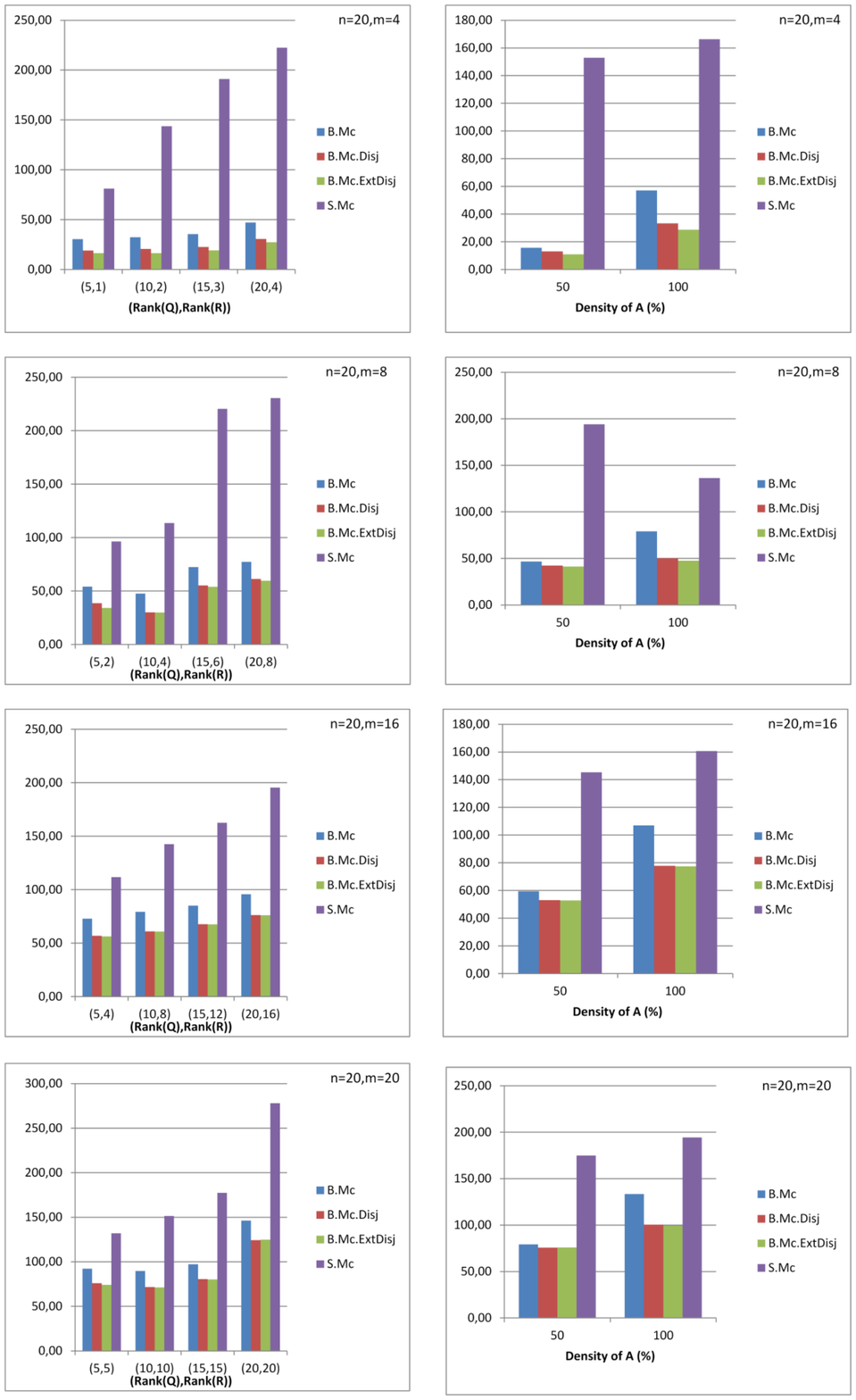}
\caption{{Relative duality gaps for box-constrained nonconvex quadratic problem ($n=20$)} \label{plots20} }
\end{figure}

\begin{figure}
\centering
 \hspace*{-1.5cm}
 \includegraphics[height=21.5cm, width=16.5cm]{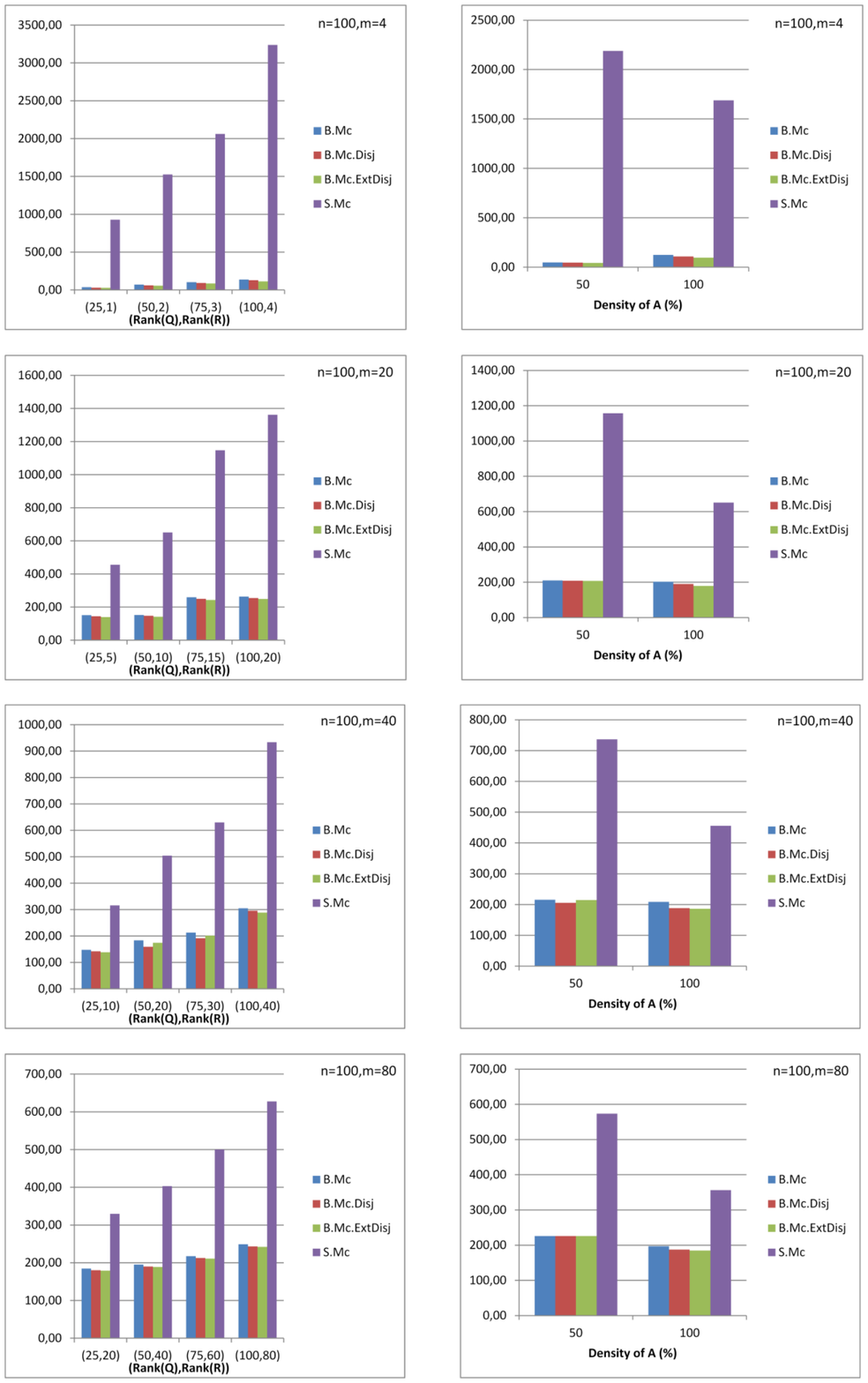}
\caption{{Relative duality gaps for box-constrained nonconvex quadratic problem ($n=100$)} \label{plots100} }
\end{figure}

From the plots in Figures \ref{plots20} and \ref{plots100}, we make the following observations.
\begin{itemize}
\item We confirm the superiority of the quadratic relaxation (\textbf{B.Mc}) over the linear relaxation (\textbf{S.Mc}) for all the instance configurations. We observe that when the ranks of the positive-semidefinite  matrices increase, this superiority also increases,  indicating that for larger ranks it is even more effective to keep the  quadratic information of the objective function in the relaxation. Also, as expected, the density of $A$ does not interfere much with the superiority of the quadratic relaxation, once its related bilinear function $x'Ay$, is linearized in both relaxations in a similar way.
\item We observe that both cutting-plane algorithms are effective in increasing the lower bound obtained by solving (\textbf{B.Mc}), for all instance configurations. In general, the improvement in the lower bounds is not influenced by the rank of the square matrices $Q$ and $R$, as the disjunctive cuts act to strengthen the relaxation of the non-symmetric lifting $W=xy'$, and therefore their effectiveness is more sensitive to modifications on the non-symmetric matrix $A$. As expected, the cuts are  more effective when $A$ is 100\% dense. In this case, more terms in the bilinear function are being relaxed, and the disjunctive cuts have more opportunity to act.
\item We confirm a slight superiority of the lower bounds obtained by the cutting-plane algorithm  \textbf{B.Mc.ExtDisj} over \textbf{B.Mc.Disj}. Again, this superiority is not influenced by the rank of the square matrices, but when $A$ is fully dense, as the cuts get more effective,  the superiority of  \textbf{B.Mc.ExtDisj} becomes more evident.
\end{itemize}

Finally, after observing the superiority of the cutting-plane algorithm \textbf{B.Mc.ExtDisj} over \textbf{B.Mc.Disj}, we did one last experiment, where we did not limit the number of cuts added in the model, but instead we limited the computational time to 3600 seconds. We ran these tests only for  instances where $A$ has density of 100\%. The goal is to see how much the algorithm can tighten the bound in this favorable case if more time is given to the algorithm. Table \ref{tab:tab3} shows average results for these instances, specifically, the duality gap given by the relaxation  (\textbf{B.Mc}), the duality gap given by \textbf{B.Mc.ExtDisj} after 3600 seconds of execution, the number of cuts added to the relaxation, and the gap closed by the disjunctive cuts, computed by
$$\mbox{Gap closed}:=\frac{\bar{z}-lb(\textbf{B.Mc.ExtDisj})}{\bar{z} - lb(\textbf{B.Mc})}\times 100\%.$$

\begin{table}[ht]
\centering
\begin{tabular}{cc|cc|cc}
\hline
&&\multicolumn{2}{c|}{Relative Duality}&\multicolumn{1}{c}{Number}&\multicolumn{1}{c}{Gap}\\
&&\multicolumn{2}{c|}{Gap (\%)}&\multicolumn{1}{c}{of cuts}&\multicolumn{1}{c}{closed}\\
$n$&$m$& B.Mc & B.Mc.ExtDisj& added& (\%)  \\ \hline
20&4&57.07&	11.81	&851.75	&80.48\\
20&8&79.15&	17.45	&736.00	&78.25\\
20&16&106.97&	49.04&	456.00	&54.97\\
20&20&133.43&	75.30	&343.00	&44.02\\
100&4&125.20&	71.64	&296.00	&46.92\\
100&20&202.95&	168.63&	166.00	&16.97\\
100&40&208.76&	180.99&	110.00&	13.56\\
100&80&197.01	&182.82	&66.00	&7.21\\
\hline
\end{tabular}
\caption{Effect of disjunctive cuts in 3600 seconds, when $A$ has density of 100\%}\label{tab:tab3}
\end{table}

From the results in Table \ref{tab:tab3}, we see that the cuts get more expensive to calculate, and the subproblems become more expensive to solve as the size of the instances increases, and therefore the number of cuts added within the time limit decreases. On the other hand, we observed that for the great majority of instances, the algorithm stopped only at the time limit, showing that all cuts generated were effective to cut off the current solution of the relaxation. Finally, we see that when a large number of cuts is added, as for the smallest instances, the reduction in the duality gap is significant.

%\begin{table}[ht]
%\centering
%\begin{tabular}{cccccc|ccc|cc}
%&&&&&&\multicolumn{3}{c|}{Bilinear Approach}&\multicolumn{2}{|c}{Symmetric Approach}\\
%%\hline
%I&n&m& r(Q)&r(R)&d(A)&B.MC & B.MC.D & B.MC.extD & S.MC & S.MC.D\\ \hline
%1&20&4&15&3&80&&&&&\\
%\end{tabular}
%\caption{Upper bounds for box constrained nonconvex quadratic problem, with different relaxations}\label{tab:bounds}
%\end{table}
%
%\begin{table}[ht]
%\centering
%\begin{tabular}{cccccc|ccc|cc}
%&&&&&&\multicolumn{3}{c|}{Bilinear Approach}&\multicolumn{2}{|c}{Symmetric Approach}\\
%%\hline
%I&n&m& r(Q)&r(R)&d(A)&B.MC & B.MC.D & B.MC.extD & S.MC & S.MC.D\\ \hline
%1&20&4&15&3&80&&&&&\\
%\end{tabular}
%\caption{Upper bounds for box constrained nonconvex quadratic problem, with different relaxations}\label{tab:times}
%\end{table}

\section{Further analysis}
\subsection{Saxena et al.'s disjunctive cuts: bilinear vs. symmetric approach}

In this section, we present a comparative analysis between the non-convex valid inequalities that generate Saxena et al.'s disjunctive cuts, considering both approaches discussed in this paper: the standard symmetric approach, where we use the symmetrization in \eqref{symmetrize}, and the alternative bilinear approach presented in \S\ref{sec:Saxena2}. The idea is to compare the strength of the inequalities, when derived from solutions of equal quality, to relaxations  (\textbf{S.Mc}) and  (\textbf{B.Mc}).

\begin{theorem}
Let $\hat{x},\hat{y},\hat{W}$ be a solution to (\textbf{B.Mc}).
Let  $u,v$ be the  left- and right-singular vectors
of the matrix $\hat{W}-\hat{x}\hat{y}'$, corresponding to one of its singular-values $\sigma$, and assume that $\sigma$ is positive. Define
\begin{align}\label{eigenvectorsym}
		z:=\frac{1}{\sqrt{2}}\left(\begin{array}{c}
								         u \\
									   	   v
										   \end{array}\right)~.
\end{align}			

If we add the positive-semidefiniteness constraints, $X-xx'\succeq 0$ and $Y-yy'\succeq 0$, to (\textbf{S.Mc}), where $X$ and $Y$ are the principal sub-matrices of $H$, formed respectively by the first $n$ and last $m$ rows and columns of $H$, and $x$ and $y$ are the sub-vectors of $h$, composed respectively by its first $n$ and last $m$ components, then, the non-convex valid inequality
\begin{equation}
\label{vineq}
\left\langle z z', H \right\rangle - (z'h)^2 \leq 0~,
\end{equation}
for (\textbf{S.Mc}), is at least as strong as the non-convex valid inequality
\begin{equation}
\label{svdconst1}
\left\langle uv', W\right\rangle - (u'x)(v'y) \leq 0~,
\end{equation}
for (\textbf{B.Mc}).

\end{theorem}

\begin{proof}
Let
\begin{align}\label{matsym2}
		\hat{H}:=\left(\begin{array}{cc}
										   \hat{x}\hat{x}' & \hat{W} \\
										   \hat{W}' & \hat{y}\hat{y}'
										   \end{array}\right), \quad  \hat{h}:=\left(\begin{array}{c}
								         \hat{x} \\
									   	   \hat{y}
										\end{array}\right) ~.
\end{align}	
	
We note that $\hat{H},\hat{h}$ is a feasible solution to (\textbf{S.Mc}) with  objective function value equal to the objective function value of (\textbf{B.Mc}) at the given solution $\hat{x},\hat{y},\hat{W}$.

Moreover, it is straightforward to verify that if $\sigma$ is the $k^{th}$ largest singular value of $(\hat{W}-\hat{x}\hat{y}')$, for some $k\in\{1,\ldots,\min\{n,m\}\}$, then $\lambda = \sigma$ is the  $k^{th}$ largest eigenvalue of $(\hat{H}-\hat{h}\hat{h}')$, with corresponding eigenvector $z$.

It is then possible to see that \eqref{vineq} and \eqref{svdconst1} are non-convex valid inequalities, equally violated by the solutions $\hat{H},\hat{h}$, of (\textbf{S.Mc}), and  $\hat{x},\hat{y},\hat{W}$, of (\textbf{B.Mc}), respectively.

Let us consider the symmetrization of the objective function of (\textbf{B.Mc}), by defining
\begin{align}\label{matsym1}
		\Gamma:=\left(\begin{array}{cc}
										   Q & \frac{1}{2} A \\
										   \frac{1}{2} A' & R
										   \end{array}\right)~, \,\, a_h:=\left(\begin{array}{c}
								         a_x \\
									   	   a_y
										   \end{array}\right)~, \,\, b_h:=\left(\begin{array}{c}
								         b_x \\
									   	   b_y
										   \end{array}\right)~,
\end{align}								       														
and using the symmetric lifting
 \begin{align}\label{symlift1}
H= hh', \, \mbox{where} \, h:=\left(\begin{array}{c}
								         x \\
									   	   y
										   \end{array}\right)~.
\end{align}

Let us also consider the decomposition of $H$ as
\begin{align}\label{matsymvar}
		H:=\left(\begin{array}{cc}
										   X & W \\
										   W' & Y
										   \end{array}\right)~, \, \mbox{where} \, X:=xx' \, \mbox{and} \, Y:=yy'~.
\end{align}

Now, using \eqref{eigenvectorsym}, \eqref{symlift1}, and \eqref{matsymvar}, we may rewrite \eqref{vineq} as
\begin{equation}
\label{svdconst2}
\left\langle uu', X\right\rangle+2\left\langle uv', W\right\rangle + \left\langle vv', Y\right\rangle   - (u'x + v'y)^2 \leq 0~,
\end{equation}
or, equivalently, as
\begin{equation}
\label{svdconst3}
\begin{array}{lcl}
\left\langle uv', W\right\rangle - (u'x)(v'y) & \leq & \frac{1}{2} \left( (u'x)^2  + (v'y)^2 -\left\langle uu', X\right\rangle - \left\langle vv', Y\right\rangle \right)~\\
&= & \frac{1}{2} \left(
 u'(xx'-X)u +  v'(yy'-Y)v \right)\\
&\leq & 0~,
\end{array}
\end{equation}
when $X-xx'\succeq 0$, and  $Y-yy'\succeq 0$.

We can then conclude that the valid inequality \eqref{vineq} is at least as strong as \eqref{svdconst1}, if we employ the positive-semidefiniteness constraints.
\end{proof}

Although we have proven that the use of positive-semidefiniteness inequalities leads to non-convex constraints, which generate disjunctive Saxena et al.'s cuts in the symmetric context, at least as strong as the non-convex constraints used in the non-symmetric context, in this work, we are interested in avoiding positive-semidefiniteness  constraints. To complement our analysis, we ran some numerical experiments comparing the strength of a \emph{single disjunctive cut} computed using Saxena et al.'s idea, when applied in the non-symmetric and  symmetric contexts (not employing positive-semidefiniteness inequalities). To construct the cuts from solutions of equal quality for  (\textbf{S.Mc}) and (\textbf{B.Mc}), we considered $Q$ and $R$ as null matrices on these tests. Except for $Q$ and $R$, all the remaining data were taken from the instances considered on the experiments described in \S\ref{sec:exper}. We have computed the average gap closed by both cuts for instances with $n=20$, and fully dense matrices $A$. We note that for $Q=0$ and $R=0$, the bounds obtained by the relaxations (\textbf{S.Mc}) and (\textbf{B.Mc}) were indeed the same for every instance considered. We also note that in the symmetric approach, the cut was not able to improve the bound for any instance, while in the non-symmetric approach it was able to improve it for every instance considered. The average percentage gap closed in relation to the solution of (\textbf{B.Mc}), for $m=4,8,16,20$, was respectively, $1.8,2.9,1.9,2.0$.

\subsection{Bilinear approach: Saxena et al.'s vs. Extended McCormick disjunctive cuts}
Now, we present a comparative analysis between the valid inequalities that generate  disjunctive cuts, when considering both non-symmetric approaches discussed in this paper,  presented in \S\ref{sec:Saxena2} and \S\ref{sec:McInstead}.
\begin{theorem}
\label{bsaxvsbmc}
Let $\hat{x},\hat{y},\hat{W}$ be a solution to (\textbf{B.Mc}).
Let  $u,v$ be the  left- and right-singular vectors
of the matrix $\hat{W}-\hat{x}\hat{y}'$, corresponding to one of its singular-value $\sigma$, and assume that $\sigma$ is positive.
Let $p_1$ and $p_2$ be defined as in \eqref{defining}, and  $[a_i,b_i]$, be bounds for $p_i$ ($i=1,2$).
Summing inequalities ($\mbox{I.1}'$) and ($\mbox{I.2}'$), we obtain:
\begin{align}
		& \left \langle W,uv' \right \rangle\leq \frac{a_2+b_2}{2} p_1 + \frac{a_1+b_1}{2} p_2  - \frac{a_1 b_2+a_2 b_1}{2}~.  \label{addmc}
\end{align}
Now, consider the definitions for  $q_1$ and $q_2$, in \eqref{qjdefinition}, and the definition for $r$, in \eqref{rdefinition}. Replacing $-q_1^2$ in \eqref{univ1}, by its secant, defined in the interval derived from the bounds on $p_1$ and $p_2$, we obtain:
\begin{equation}
\begin{array}{lcl}
\label{saxmf}
		 \left \langle W,uv' \right \rangle &\leq& \displaystyle \frac{a_1+b_1+a_2+b_2}{4} p_1 ~+~  \frac{a_1+b_1+a_2+b_2}{4} p_2  \\
		 & &-~  \displaystyle \frac{a_1 b_2+ a_2 b_1 +a_1 b_1+ a_2 b_2}{4} ~-~ \displaystyle \left(\frac{p_1-p_2}{2}\right)^2 ~.		
\end{array}
\end{equation}
Then,
\begin{enumerate}
\item[($i$)] When $p_1=p_2$ (case where $n=m$, $A$ is symmetric, and $x=y$), both inequalities, \eqref{addmc} and \eqref{saxmf}, are equivalent.
\item[($ii$)] When $p_1 \neq p_2$  (general case),
\begin{enumerate}
\item if $b_1-a_1 = b_2-a_2$, constraint  \eqref{saxmf} dominates  \eqref{addmc},
\item otherwise, neither of the constraints, \eqref{addmc} or \eqref{saxmf}, dominates  the other.
\end{enumerate}
\end{enumerate}
\end{theorem}
\begin{proof}
The result in ($i$) can be easily verified by replacing $p_1$ and $p_2$ by $p:= p_1=p_2$ in  inequalities \eqref{addmc} and \eqref{saxmf}. They are then, both reduced to
\begin{align}
		 &\left \langle W,uu' \right \rangle  \leq (a+b) p  - ab~,  \label{addmcsym}
\end{align}
where we have the secant of $p^2$ in the right-hand side of the inequality, defined in the interval $[a,b]$,  where $a:=a_1=a_2$, and $b:=b_1=b_2$.

To analyze  the  general  case in ($ii$), we rewrite \eqref{saxmf} as
\begin{align}
		\left \langle W,uv' \right \rangle &\leq \frac{a_1+b_1+a_2+b_2}{4} p_1 + \frac{a_1+b_1+a_2+b_2}{4} p_2 \\
		&\quad -  \frac{a_1 b_2+ a_2 b_1 +a_1 b_1+ a_2 b_2}{4} - \left(\frac{p_1-p_2}{2}\right)^2 \\
		%& = \frac{a_2+b_2}{2} p_1 + \frac{a_1+b_1}{2} p_2  - \frac{a_1 b_2+a_2 b_1}{2}\\
		%& + \frac{a_1+b_1-a_2-b_2}{4} p_1 +  \frac{a_1+b_1-a_2-b_2}{4} p_2 \\
		%&-  \frac{-a_1 b_2- a_2 b_1 +a_1 b_1+ a_2 b_2}{4} - \left(\frac{p_1-p_2}{2}\right)^2 \\
    %&   = \frac{a_2+b_2}{2} p_1 + \frac{a_1+b_1}{2} p_2  - \frac{a_1 b_2+a_2 b_1}{2}\\
		%& + \frac{a_1+b_1-a_2-b_2}{4} p_1 +  \frac{a_1+b_1-a_2-b_2}{4} p_2 \\
		%& -  \frac{-a_1 a_2- b_2 b_1 +a_1 b_1+ a_2 b_2}{4} - \left(\frac{p_1-p_2}{2}\right)^2 \\
		%& - \frac{-a_1 b_2- a_2 b_1 +a_1 a_2+ b_1 b_2}{4}\\	
		&   = \frac{a_2+b_2}{2} p_1 + \frac{a_1+b_1}{2} p_2  - \frac{a_1 b_2+a_2 b_1}{2}  \label{eqi}\\
		& \quad + \frac{(a_1-b_2)+(b_1-a_2)}{2} \left(\frac{p_1-p_2}{2}\right)    -  \frac{(a_1 - b_2)( b_1 - a_2)}{4}   - \left(\frac{p_1-p_2}{2}\right)^2  \label{secant1}\\
		& \quad - \frac{(b_1-a_1)(b_2- a_2)}{4}~. \label{eq3}
\end{align}

We note that the expression in \eqref{eqi} is equivalent to the right-hand side of  \eqref{addmc}.

When $p_1,p_2=a_1,b_2$, or $p_1,p_2=b_1,a_2$, the value of the secant in the first two terms of \eqref{secant1}
is equal to the value of the quadratic in the last term,  and therefore, the sum in (\ref{secant1}-\ref{eq3})  is negative.

When $p_1=(a_1+b_1)/2$ and $p_2=(a_2+b_2)/2$, the sum in (\ref{secant1}-\ref{eq3})  assumes its maximum value and becomes
%we have
%\begin{align*}
    %&  \frac{(a_1-b_2)+(b_1-a_2)}{2} \left(\frac{p_1-p_2}{2}\right) -  \frac{(a_1 - b_2)( b_1 - a_2)}{4} - \left(\frac{p_1-p_2}{2}\right)^2 \\
		%%& = \frac{(\frac{(a_1-b_2)}{2}-\frac{(b_1-a_2)}{2})^2}{4}\\
		%%&= \frac{(a_1-b_1+a_2-b_2)^2}{16}\\
		%& \geq \frac{(b_1-a_1)(b_2- a_2)}{4}~,
%\end{align*}
%because
%\begin{align*}
\begin{equation*}
    \frac{(a_1-b_1+a_2-b_2)^2}{16} - \frac{(b_1-a_1)(b_2- a_2)}{4}
		~=~ \frac{((a_1-b_1)-(a_2-b_2))^2}{16}~.
\end{equation*}
%\end{align*}

Therefore, if $b_1-a_1=b_2-a_2$, the sum in (\ref{secant1}-\ref{eq3}) is
equal to zero, and constraint \eqref{saxmf}  dominates \eqref{addmc}. Otherwise, the sum in (\ref{secant1}-\ref{eq3}) is positive. In this case, neither of the constraints, \eqref{addmc} or \eqref{saxmf}, dominates  the other. \end{proof}

We note that the inequality obtained by adding ($\mbox{I.3}'$) and ($\mbox{I.4}'$) can also be compared to \eqref{univ2} with a similar analysis to the one in Theorem \ref{bsaxvsbmc}.

To complement our final analysis, we ran some numerical experiments comparing the strength of a \emph{single disjunctive cut} computed in three different ways: ($i$) using the ideas presented in  \S\ref{sec:Saxena2} (B.Mc.Disj), ($ii$) using the ideas presented in  \S\ref{sec:McInstead} (B.Mc.ExtDisj), ($iii$) mixing both ideas, i.e., considering all valid inequalities used in the two first approaches to model the CGLP that generates the disjunctive cut.  We present in Table \ref{tab:tab4}, the average gap closed in relation to the solution of (\textbf{B.Mc}), by the three disjunctive cuts, for the instances with $n=20$ and fully dense matrices $A$, that were used on the experiments described in \S\ref{sec:exper}.  We note that, although we have very similar results for the different cuts, we have different winners for each group of instances. For $m=4$, using all valid inequalities together generated the best average bound,
 for $m=8$ and $16$, the winner was the disjunctive cut described in \S\ref{sec:McInstead} (B.Mc.Disj), and for $m=20$, the winner was the disjunctive cut described in \S\ref{sec:Saxena2}  (B.Mc.ExtDisj). The results support the conclusion in Theorem \ref{bsaxvsbmc}, and also show that using the valid inequalities described in \S\ref{sec:Saxena2} and  \S\ref{sec:McInstead}  together, we can tighten the bound obtained when using the inequalities separately.

\begin{table}[ht]
\centering
\begin{tabular}{cc|ccc}
\hline
&&\multicolumn{3}{c}{Gap closed (\%)}\\
$n$&$m$&  B.Mc.Disj & B.Mc.ExtDisj & Mixed ideas \\ \hline
20&4&1.89&1.95&2.33\\	
20&8&3.00&2.52&2.61\\
20&16&2.12&2.06&2.02\\
20&20&2.09&2.10&2.06\\
\hline
\end{tabular}
\caption{Effect of a single disjunctive cut generated by different valid inequalities}\label{tab:tab4}
\end{table}

\section{Conclusions} We believe that there is a lot of potential for balancing the algebraic and structural extraction of convexity from non-convex polynomial functions. We believe that the best choice, in many circumstances, may not be one or the other, but rather a good harmonization of both viewpoints. We hope that our work provokes more in this direction.

\bibliographystyle{plain}
\bibliography{fl_mixed}

\end{document}